\documentclass{amsart}
\usepackage[nobysame]{amsrefs}
\usepackage{amsmath}
\usepackage{amssymb}
\usepackage{graphicx}
\usepackage{ulem}
\usepackage{xcolor}

\theoremstyle{plain}
\newtheorem{theorem}{Theorem}[section]
\newtheorem{lemma}{Lemma}[section]
\newtheorem{corollary}{Corollary}[section]
\newtheorem{proposition}{Proposition}[section]

\newtheorem*{claim*}{Claim}
\newtheorem*{lemma*}{Lemma}
\newtheorem*{theorem*}{Theorem}

\theoremstyle{definition}
\newtheorem{definition}{Definition}[section]
\newtheorem{example}{Example}[section]

\theoremstyle{remark}
\newtheorem{remark}{Remark}
\newtheorem*{remark*}{Remark}

\newtheorem*{conjecture*}{Conjecture}

\begin{document}

\title[Analysis of o.a.p \& the G-K conjecture]{Analysis of orbit accumulation points and the Greene-Krantz conjecture}

\author{Bingyuan Liu}
\address{Department of Mathematics, Washington University, Saint Louis, USA}
\email{bingyuan@math.wustl.edu}



\date{July 22, 2014}



\begin{abstract}
In $\mathbb{C}^2$, we classify the domains for which $\rm Aut(\Omega)$ is noncompact and describe these domains by their defining functions. This note is based on the technique of the scaling method introduced by Frankel \cite{Fr86} and Kim \cite{Ki90}. One feature of this article is that we are able to analyze the defining functions of infinite type boundary. As a corollary, we also prove a result that under some conditions, $\rm Aut(\Omega)$ contains $\mathbb{R}$, which is an extension of \cite{Fr86}. 
\end{abstract}

\maketitle

\section{Introduction}

We call a connected open subset in $\mathbb{C}^2$ a domain, and by an automorphism we mean a holomorphic automorphism. 

Let $\mathbb{D}$ denote the unit disc, and let $\mathbb{H}^\pm$ denote the upper-half (lower-half) plane of $\mathbb{C}$. We use $\Im z, \Re z$ to denote the imaginary and real part of $z$. Also, we will not distinguish convergence and subsequence convergence. We denote the operator norm of a matrix by $\|\cdot\|_{op}$. We also denote $w=u+iv$ sometimes.

When a bounded convex domain $\Omega$ in $\mathbb{C}^2$ admits a noncompact automorphism group, Frankel was able to construct in \cite{Fr86} a new (in general unbounded) domain $\Omega'$ which is biholomorphic, by the inverse of the limit of $(J\phi_j(q))^{-1}(\phi_j(z,w)-\phi_j(q))$, to $\Omega$, where $q\in\Omega$ is an arbitrary interior point. Soon after that, Kim described the domain $\Omega'$ in \cite{Ki90} by the defining function. His argument is based on the fact that the boundary of $\Omega$ is invariant under automorphism and that the boundary of $\Omega'$ must exist (otherwise, $\Omega'$ can not be hyperbolic, i.e., $\Omega'$ can not be equivalent to a bounded domain in $\mathbb{C}^2$). Moreover, it is enough to use the local defining function around the accumulation points of $\Omega$ to analyze the boundary of $\Omega'$ because other parts of boundary will be transformed, roughly speaking, to $\infty$. 

Pinchuk's scaling method (see \cite{BP91}) is also well-known. His method directly involves the boundary during the procedure of construction for the final biholomorphism. Although his method is also interesting, we will not treat it further here.

It is well-known that in $\mathbb{C}^2$, possibly after a global biholomorphic transform $F$, an arbitrary domain with a boundary point $p\in\partial\Omega$ can be defined by the defining function  $\Im w>\rho(z,\bar{z}, \Re w)$ locally around $(0,0)$, where $\rho(z,\bar{z},\Re w)=O(|z|^2)+O((\Re w)^2)+O(z\Re w)$ and $p$ has been translated to $(0,0)$.

\begin{definition}[normal domain at a point]
Let $p\in\partial\Omega$. If there exists a neighborhood $U$ of $p$ in $\mathbb{C}^2$ such that after the global transform $F$ as above, $F(\Omega\cap U)=\lbrace (z,w): \Im w>\rho(z,\bar{z}, \Re w) \rbrace$, where locally $\rho(z,\bar{z},\Re w)=\rho(z,\bar{z},0)+O((\Re w)^2)+O(z\Re w)$, then we call $\Omega'=F(\Omega)$ a normal domain of $\Omega$ at $p$.
\end{definition}

In $\mathbb{C}^2$, the bidisc $\mathbb{D}\times\mathbb{D}$ has been intensely studied in the field of several complex variables and the fact that it is not biholomorphic to the ball can be dated back to the period of Poincar\'{e}. However, the following type of domain is less well-known.

We denote $\Omega=\lbrace (z,w)\in\mathbb{C}^2: z\in\mathbb{D}, w\in e^{i\theta(z)}\mathbb{H}^+\rbrace$, where $\theta$ is a real continuous function of $z$, by $\mathbb{D}\rtimes_\theta\mathbb{H}^+$. One can see that, if $\theta$ is a zero function, then it is biholomorphic to the bidisc, but in general this is not the case. Roughly speaking, this type of domain is a fiber bundle with the base $\mathbb{D}$, and the fiber above each point of the base is a rotation of $\mathbb{H}^+$ by an angle determined by the base point. More generally, for an arbitrary domain $\mathfrak{I}\subset\mathbb{C}$ we denote $\mathfrak{I}\rtimes_\theta\mathbb{H}:=\lbrace (z,w): z\in\mathfrak{I}, w\in e^{i\theta(z)}\mathbb{H}^+ \rbrace$, where $\theta$ is a real continuous function of $z\in\mathfrak{I}$.

\begin{example}
$\mathbb{D}\rtimes_{2\Re}\mathbb{H}^+=\lbrace (z,w)\in\mathbb{C}^2: |z|<1, w\in e^{2i\Re z}\mathbb{H}^+\rbrace$\\
\end{example}

We denote the automorphism group by $\rm Aut(\Omega)$. This is the set containing all holomorphic automorphisms of $\Omega$ endowed with the operation of composition.

The automorphism group of $\mathbb{D}\rtimes_{2\Re}\mathbb{H}^+$ is not compact. Indeed, for an arbitrary natural number $n$ let $L_n:=(z, \dfrac{w}{n})$. One can see $L_n\in\rm Aut(\mathbb{D}\rtimes_{2\Re}\mathbb{H}^+)$ for arbitrary $n$, and there is an interior point $q$ such that $L_n(q)\rightarrow p$ for some $p\in\partial(\mathbb{D}\rtimes_{2\Re}\mathbb{H}^+)$ as $n\rightarrow\infty$.
\begin{remark}
The reader should be warned, $\mathbb{D}\rtimes_\theta\mathbb{H}^+$ defined here in general has a nonsmooth boundary (at most piecewise smooth). Indeed, it is Levi-flat for each smooth piece. We will use the scaling method to obtain a biholomorphism between some bounded domains and $\mathbb{D}\rtimes_\theta\mathbb{H}^+$. However, there are no clues on whether any bounded domain with a (globally) smooth boundary is biholomorphic to it. This is one of the key points in Greene-Krantz conjecture. For the discussion, please see Section \ref{sec3}.
\end{remark}

Since the note will frequently mention the concept ``finite type'' (in sense of D'Angelo), we will define it here briefly. The interested reader is referred to \cite{D'93}.

\begin{definition}
Let $X$ be a nontrivial analytic disc passing through $(0,0)$ in $\mathbb{C}^2$ defined by $\psi:\mathbb{D}\rightarrow\mathbb{C}^2$, with $\xi\mapsto(f_1(\xi), f_2(\xi))$. Let $\rho$ be the defining function of a domain $\Omega$ for which $(0,0)\in\partial\Omega$ . We say that $\Omega$ is type $t$ around $(0,0)$ if 
\begin{equation*}
\tau(\partial\Omega, (0,0)):=\sup_{X}\frac{\nu(\rho\circ\psi)}{\nu(\psi)}=t,
\end{equation*}
where $\nu(f)$ denotes the order of vanishing of $f$ at $0$. 
If $t=\infty$, $\Omega$ is said to be of infinite type around $(0,0)$, and similarly for finite type.
\end{definition}

\begin{remark}
By the definition, one can see the following immediately. Assume the defining function is $\rho=v-r(z,\bar{z},u)=v-r(z,\bar{z},0)-C(z,\bar{z})u-o(u^2)$. Then if $(0,0)$ is finite type, it means $r(z,\bar{z},0)$ has to have a finite order of vanishing at $0$ (otherwise, let $\psi=(\xi,0)$, and $\tau=\infty$). On the other hand, assume $\rho=v-r(z,\bar{z})$ is the defining function; if $(0,0)$ is infinite type, then $r$ has to be of $o(|z|^m)$ for any positive $m$, as otherwise $\tau<\infty$. 
\end{remark}

The study of infinite types is extremely hard because it is essentially the study of smooth functions that are not analytic. Since it is not very clear what happens for a non-analytic smooth function, we are unable to analyze them all.

\begin{definition}
Let $\Omega\subset\mathbb{C}^2$ be a domain and let $p\in\partial\Omega$ be a boundary point of infinite type. Assume the defining function of the normal domain of $\Omega$ at $p$ is $v=\rho(z,\bar{z},0)+\text{higer terms}$.
\begin{enumerate}
\item We say $p$ is of infinite type I, if $\rho$ satisfies the condition that $\rho(az,a\bar{z},0)\leq o(1)\rho(z,\bar{z},0)$ for all complex number $a$ such that $0<|a|<1$.
\item Otherwise, we say $p$ is of infinite type II.
\end{enumerate}
\end{definition}

Note that the condition of infinite type I implies that for any complex $0<|a|<1$, we have $\rho(az,a\bar{z},0)\leq |a|^{m_z}\rho(z,\bar{z},0)$, where $m_z\rightarrow\infty$ as $z\rightarrow 0$.
It is not hard to see that the domain $\Omega_{G-K}$ defined by $|w|^2+|z|^{\frac{1}{-|z|^2}}<1$ in \cite{GK93} is of infinite type I. This is because of direct computation after finding the normal domain of $\Omega_{G-K}$, which is $\lbrace(z,w=u+iv)\in\mathbb{C}^2: v>1-\sqrt{1-e^{-\frac{1}{|z|^2}}-u^2}\rbrace$. Furthermore, all other known bounded domains which support the Greene-Krantz conjecture (see Section \ref{sec3}) are of infinite type I. It is also interesting to ponder whether a domain of infinite type II exists in the context of basic function theory.

Among other results, the most important is the following theorem.

\begin{theorem}\label{main}
Let $\Omega$ be a bounded domain in $\mathbb{C}^2$. Assume there is a family of automorphisms $\phi_j=(f_j, g_j)$ and an interior point $q\in\Omega$ such that $\phi_j(q)\rightarrow p$, where $p\in\partial\Omega$ is not of infinite type II. We also assume the normal domain of $\Omega$ at $p$ is locally convex and has a smooth boundary around $p$. Then $\Omega$ is biholomorphic to $\mathbb{D}\rtimes_\theta\mathbb{H}^+$ of $v>\rho_k(z,\bar{z})$ where $\rho_k$ is an homogeneous polynomial with degree $k$. 
\end{theorem}

It will be very pleasant if the ``locally convex'' condition can be removed. However, it turns out that this condition is not removable in the current note with the scaling method, because the scaling will not converge otherwise. For the discussion, please see Remark \ref{nor}.

In Section \ref{sec1}, we will prove some properties about the Jacobian of one family of noncompact automorphism maps. We also extend Cartan's theorem about determinant of the Jacobian (see \cite{Na95}) to the eigenvalue functions of the Jacobian $J\phi_j$. But please note that the Cartan's theorem does not hold if we replace determinant with the norm of the matrix (see Remark \ref{nr}). We will prove Theorem \ref{main} in Section \ref{sec2} and we will add some remarks about Greene-Krantz conjecture in Section \ref{sec3}. In Section \ref{sec3} we also use the recent result of \cite{CV13} to prove the bidisc is not biholomorphic to any bounded domain with smooth boundary and finally affirm the Greene-Krantz conjecture for a special case (see Corollary \ref{GK}).

Finally, we want to point out that most of the results in this article quite possibly may hold in higher dimensions. However, the proof will be slightly different, because in higher dimension, the uniformization theorem will be unavailable and the concept of ``finite type'' will be more subtle.

\section{The general properties of $\rm Aut(\Omega)$}\label{sec1}
Let $\lbrace \phi_j\rbrace_{j=1}^\infty\subset \rm Aut(\Omega)$ be a sequence in the automorphism group of $\Omega$. We denote by $\lambda_j^i$, $i=1,2$ the two eigenvalues (functions) of the Jacobian of the biholomorphic map $\phi_j$. We also let $\lambda_0^i$, $i=1,2$ be the limit of $\lambda_j^i$ as $j\rightarrow\infty$. We order $\lambda_j^i$ with $\succ$ by using the lexicographic order from the triple $(\|\cdot\|,\Re\cdot, \Im\cdot)$. We also sometimes abuse notation by omitting $j$ when $j> 0$.

\begin{proposition}
Let $\Omega$ be a bounded domain in $\mathbb{C}^2$. Assume there is a family of automorphisms $\phi_j=(f_j, g_j)$ such that there is an interior point $q\in\Omega$ and a boundary point $p\in\partial\Omega$ with $\phi_j(q)\rightarrow p$ as $j\rightarrow\infty$. Then, $p$ is the UNIQUE orbit accumulation point if and only if $\lambda^i_j(z,w)\rightarrow 0$ as $j\rightarrow\infty$ for all $(z, w)\in\Omega$. 
\end{proposition}

\begin{proof}
Suppose we can find a vector $v$ and an interior point $a\in\Omega$ such that the pushforward $\|(\phi_j^*(a))v\|=\|J\phi_j(a)\cdot v\|>\epsilon_0$ as $j\rightarrow\infty$. Let $h$ be an analytic disc with $h'(0)=v$ and $h(0)=a$. We consider $\phi_j\circ h$, and by Cartan's theorem the image of the limit of $\phi_j\circ h$ must be contained in the boundary. Since the pushforward is never zero around $a$, the limit of the image of $\phi_j\circ h$ is non-trivial, which contradicts our assumption. The other direction is trivial by considering the rank of the Jacobian of the limit map $\displaystyle\phi_0=\lim_{j\to\infty}\phi_j$.
\end{proof}

The following lemma says that in $\mathbb{C}^2$, the eigenvalues (functions) of the Jacobian of a biholomorphism are holomorphic functions not just the determinant.

\begin{lemma}\label{st}
Let $\Omega$ be a bounded domain in $\mathbb{C}^2$. Assume there is a biholomorphic map $\phi_0=(f_0, g_0)$. Then the eigenvalue functions $\lambda^i(z,w)$ of $J\phi_0$ are both holomorphic functions for $i=1,2$. Moreover, if there is a family of automorphisms $\phi_j$ such that there exists an interior point $q$ and a boundary point $p\in\partial\Omega$ with $\phi_j(q)\rightarrow p$ as $j\rightarrow\infty$, then the eigenvalue functions $\lambda_j^i(z,w)$ of $J\phi_j(z,w)$ approach  holomorphic functions which are either zero everywhere or nowhere vanishing for $(z,w)\in \Omega$.
\end{lemma}

\begin{proof}
First, we show the eigenvalues $\lambda^i$ depend on $(z,w)$ continuously. Consider $\lambda^i(z,w)$, which satisfies the characteristic equation
\begin{equation}\label{ch}
(\lambda^i(z,w)-a_{11}(z,w))(\lambda^i(z,w)-a_{22}(z,w))-a_{12}(z,w)a_{21}(z,w)=0,
\end{equation}
where the Jacobian is $J\phi=\begin{pmatrix}
a_{11}&a_{12}\\a_{21}&a_{22}
\end{pmatrix}$. Then for any interior point $(z_0,w_0)\in\Omega$, the limit $\displaystyle\lim_{\substack{z\to z_0\\w\to w_0}}\lambda^i(z,w)$ also satisfies the equation (\ref{ch}). By the uniqueness of solutions we can see $\displaystyle\lim_{\substack{z\to z_0\\w\to w_0}}\lambda^i(z,w)=\lambda^i(z_0, w_0)$.

Without loss of generality, we will discuss $\lambda^1$ only. Let us denote $\lbrace (z, w): \lambda^1(z,w)=\dfrac{a_{11}(z,w)+a_{22}(z,w)}{2}\rbrace$ by $E$ and $\Omega\backslash E$ by $F$. By the implicit function theorem, $\lambda^1$ is holomorphic on $F$. We also observe that on $E$, $\lambda^1=\lambda^2$. Moreover, $(\lambda^1)^2$ is holomorphic on $F$, and on $E$, $(\lambda^1)^2=\lambda^1\lambda^2=a_{11}a_{22}-a_{12}a_{21}$, which is also a holomorphic function. It is obvious $E$ is a closed subset of $\Omega$, while $F$ is open. 

If $E$ does not contain any interior point, then by the theorem of removable singularities for several complex variables in \cite{Ra03}, $(\lambda^1)^2$ on $F$ can be extended to $\widetilde{(\lambda^1)^2}$ on $\Omega$ holomorphically because $E$ can be thought of as the zero set of the holomorphic function $(\dfrac{a_{11}+a_{22}}{2})^2-\lambda^1\lambda^2=(\dfrac{a_{11}+a_{22}}{2})^2-(a_{11}a_{22}-a_{12}a_{21})$. Moreover, $(\lambda^1)^2=\widetilde{(\lambda^1)^2}$ on $\Omega$ because $\lambda^1$ is continuous and $E$ has empty interior. If $E$ contains interior points, then for arbitrary $z'\in E^o$, we have $(\lambda^1)^2=\lambda^1\lambda^2=a_{12}a_{21}$, which is again holomorphic. So, $(\lambda^1)^2$ is holomorphic on $F\cup E^o$, and again by the theorem of removable singularities, we have its extension $\widetilde{(\lambda^1)^2}$ on $\Omega$ by continuity and the fact that there is no interior of $\Omega\backslash(F\cup E^o)$, we see $\widetilde{(\lambda^1)^2}=(\lambda^1)^2$.

Now we prove $\lambda^1$ is holomorphic. For this, we take the derivative 
\begin{equation*}
\dfrac{\partial}{\partial\bar{z}}((\lambda^1)^2)=\dfrac{\partial}{\partial\bar{w}}((\lambda^1)^2)=0
\end{equation*}
because of holomorphicity. We obtain
\begin{equation*}
\lambda^1\dfrac{\partial}{\partial\bar{z}}(\lambda^1)=\lambda^1\dfrac{\partial}{\partial\bar{w}}(\lambda^1)=0.
\end{equation*}
However, this implies $\dfrac{\partial}{\partial\bar{z}}(\lambda^1)=\dfrac{\partial}{\partial\bar{w}}(\lambda^1)=0$, because otherwise the determinant of the Jacobian is zero somewhere, which contradicts biholomorphicity.

We prove the second statement by Hurwitz's theorem for $\lambda_j^1$ and $\lambda_j^2$. For the second argument, we can prove it by Hurwitz theorem for $\lambda^1_j$ and $\lambda^2_j$. (Note $\lambda^i_j$ is nowhere zero otherwise, $det(J\phi_j(z,w))$ is zero somewhere which is impossible because $\phi_j$ is automorphism.) Hence, we just need to check the uniform boundedness of $\lambda_j^i$ on arbitrary closed subsets of $\Omega$. Indeed, both of $\lambda^1_j\lambda^2_j=\det(J\phi_j)$ and $\lambda^1_j+\lambda^2_j=\dfrac{\partial f_j}{\partial z}+\dfrac{\partial g_j}{\partial w}$ are uniformly bounded on arbitrary compact subsets by Cauchy estimates ($\Omega$ is bounded). If $\lambda^1_j$ is not uniformly bounded on a compact subset then there is a sequence $\lbrace(z_j,w_j)\rbrace\Subset\Omega$ such that $|\lambda^1_j(z_j, w_j)|\rightarrow\infty$ as $j\rightarrow\infty$. By uniform boundness of $\lambda^1_j\lambda^2_j$ on compact subsets, we can see $|\lambda^2(z_j,w_j)|\rightarrow 0$, but this contradicts the fact that $\lambda^1_j+\lambda^2_j$ is also uniformly bounded on compact subsets. 
\end{proof}

In $\mathbb{C}^2$ the discussion above tells us only two cases can happen given a noncompact automorphism group:
\begin{enumerate}
\item (Orbit Accumulation Point Case) The image of $\phi_0=\lim_{j\to\infty}\phi_j$ contains just one point after passing to subsequences. Both of eigenvalue functions $\lambda^i_j$ of $J\phi_j$ approach to $0$.
\item (Orbit Accumulation Variety Case) The image of $\phi_0=\lim_{j\to\infty}\phi_j$ contains a (regular) one-dimensional complex variety passing to subsequences. Only one of eigenvalue functions $\lambda^i_j$ of $J\phi_j$ approaches to $0$.
\end{enumerate}

The next proposition, which is obtained from the previous discussion, is a generalization of Cartan's theorem in \cite{Na95} (Cartan's theorem states that the limit of automorphisms $\phi_j$ of a bounded domain $\Omega$ is still an automorphism if and only if the determinant of $J\phi_j$ does not converge to $0$). 

\begin{proposition}
Let $\Omega$ be a bounded domain in $\mathbb{C}^2$. Assume there is a family of automorphism $\phi_j=(f_j, g_j)$. Then the limit of $\phi_j$ (in the compact-open topology) is still an automorphism if and only if neither of the eigenvalues $\lambda_j^i$ of $J\phi_j$ converges to $0$ in the topology of uniform convergence on compact subsets.
\end{proposition}

\begin{remark}\label{nr}
The readers should be warned, in general, that the following imitation of Cartan's theorem is not true: Given a family of automorphism groups $\phi_j$ so that $\phi_j(q)\rightarrow p$ as $j\rightarrow\infty$ where $q\in\Omega$ but $p\in\partial\Omega$, then $\| J\phi_j(q)\|_{op}\rightarrow 0$. Please see the following counterexample. However, the reader can easily show that, under the assumption that the automorphism sequence is in the orbit accumulation point case (not the  orbit accumulation variety case), $\|\phi_j(q)\|_{op}\rightarrow 0$ still holds, for an arbitrary interior point $q\in\Omega$.
\end{remark}

\begin{example}
Let $\Omega$ be the bidisc with radius $(1,1)$, centered at the origin, and $\phi_j(z,w)=(z, \dfrac{w-\alpha_j}{1-\bar{\alpha_j}w})$, where $\alpha_j\rightarrow 1$ as $j\rightarrow\infty$. The Jacobian of each automorphism has an entry of $1$, so $\| J\phi_j(q)\|_{op}\not\rightarrow 0$. 
\end{example}

\begin{proposition}\label{uni}
Let $\phi_j$ be as in Proposition \ref{st}. Then $\det((J\phi_j(q))^{-1}J\phi_j(z,w))$ is locally uniformly bounded below and above by positive number.
\end{proposition}
\begin{proof}
Let $\bar{D}\subset\Omega$ be a closed neighborhood of an interior point $q\in\Omega$. Consider $\det(J\phi_j(z,w))$, which is bounded for each $j$, so there is a $q_j\in\bar{D}$ such that 
\begin{equation*}
|\det(J\phi_j(q_j))|=\max_{(z,w)\in\bar{D}}|\det(J\phi_j(z,w))|.
\end{equation*}
Hence $\det((J\phi_j(q_j))^{-1}J\phi_j(z,w))$ is a normal family because it is uniformly bounded by $1$. Again, by Hurwitz's theorem, $\det((J\phi_j(q_j))^{-1}J\phi_j(z,w))$ converges uniformly on $\bar{D}$ to a nowhere zero holomorphic function $c_0(z,w)$, because otherwise, the sequence $\det((J\phi_j(q_j))^{-1}J\phi_j(q_j))$ approaches $0$, which is impossible. Specifically, the sequence $\det((J\phi_j(q_j))^{-1}J\phi_j(q))$ approaches a nonzero number $c_0$ and thus $\det(J\phi_j(q_j)(J\phi_j(q))^{-1})$ approaches $\dfrac{1}{c_0}$. We observes that $\det((J\phi_j(q))^{-1}J\phi_j(z,w))$ approaches a nowhere zero holomorphic function $c(z,w)=\dfrac{c_0(z,w)}{c_0}$, which completes the proof.
\end{proof}

With a similar proof to that of Proposition \ref{uni}, we can show the following
\begin{proposition}
$\dfrac{\lambda_j^i(z,w)}{\lambda_j^i(q)}$ is locally uniformly bounded below and above by positive constants for $i=1,2$.
\end{proposition}

\section{The proof of Theorem \ref{main}}\label{sec2}

For information on the Hausdorff metric, we refer readers to a nice survey \cite{He99}.

Let $H$ be the regular biholomorphic mapping from $\mathbb{C}^2$ into $\mathbb{CP}^2$ (endowed with the Fubini-Study metric), that maps $(z,w)\in\mathbb{C}^2$ to $[1,z,w]\in\mathbb{CP}^2$. Suppose we have a family of open subsets $\Omega_j\subset\mathbb{C}^2$ such that $\cap_j \Omega_j\neq\varnothing$. We will define the open subset $\widehat{\Omega}$ of $\mathbb{C}^2$ as the limit of the given family of open subsets $\lbrace\Omega_j\rbrace$ in $\mathbb{C}^2$. Since $\mathbb{CP}^2$ is a complex manifold with finite diameter, the closure $\overline{H(\Omega_j)}$ is a Cauchy sequence with Hausdorff metric on nonempty closed bounded subsets of $\mathbb{CP}^2$, because the total boundedness of $\mathbb{CP}^2$ implies the total boundedness of the Hausdorff metric. Moreover, due to the completeness of $\mathbb{CP}^2$, the Cauchy sequence $\overline{H(\Omega_j)}$ has a limit which is also a nonempty closed bounded subset of $\mathbb{CP}^2$, and we denote it by $\widetilde{\Omega}$. Please note $\widetilde{\Omega}$ is closed, and we define $\widehat{\Omega}$ to be the subset of interior points of $\widetilde{\Omega}$.

The next lemma can be considered a generalized open mapping theorem in $\mathbb{C}^2$. For higher dimensions, a similar result also holds and is not hard to be formulated and proved.

\begin{lemma}[generalized open mapping theorem]\label{gomp}
Let $\phi$ be a holomorphic map defined on $\Omega\subset\mathbb{C}^2$, and assume the determinant of the Jacobian of $\phi$ is nowhere vanishing. Then $\phi(\Omega)$ is open.
\end{lemma} 
\begin{proof}
For any $q\in\Omega$, the determinant of the Jacobian $J\phi(q)$ is not zero. So by the inverse function theorem, there exists an open neighborhood $U_q$ of $q$ such that $\phi(U_q)$ is open. Moreover, for a cover $\lbrace U_q\rbrace$ where $q$ varies through $\Omega$, we have $\displaystyle\phi(\Omega)=\phi(\bigcup_q U_q)=\bigcup_q\phi(U_q)$ is open. 
\end{proof}

The following proposition is a generalization of the well-known scaling method of \cite{Fr86}. 

\begin{proposition}\label{s1p}
Let $\Omega$ be a bounded domain in $\mathbb{C}^2$. Assume there is a family of automorphisms $\phi_j=(f_j, g_j)$ and an interior point $q\in\Omega$ such that $\phi_j(q)\rightarrow p$ where $p\in\partial\Omega$. Let $\Omega$ be locally convex around $p=(0,0)$. Let $A_j(z,w)=(J\phi_j(q))^{-1}(\phi_j(z,w)-\phi_j(q))$. Then $A_j$ converges uniformly on compact subsets to a biholomorphism between $\Omega$ and $\widehat{\Omega}$, where $\widehat{\Omega}$ is the limit of $A_j(\Omega)$.
\end{proposition}

Before the proof of Proposition \ref{s1p}, we need several lemmas.

\begin{lemma}
If $\|(J\phi_j(q))^{-1}J\phi_j(z,w)\|_{op}$ is uniformly bounded on arbitrary compact subsets of $\Omega$, then $A_j$ converges uniformly on each compact subset in $\mathbb{CP}^2$.
\end{lemma}
\begin{proof}
For arbitrary $q\in\Omega$ we now prove that $A_j$ is generalized normal, where a family of functions is called generalized normal if $A_j$ is normal or $\|A_j(z,w)\|$ converges to $\infty$ everywhere. 

Fix $q\in\Omega$ and an arbitrary connected compact subset $D\Subset\Omega$. For any $q'\in D$ we can always find a curve $\gamma(t)$ parameterized by length so that $\gamma(0)=q$ and $\gamma(1)=q'$. We consider $\|A_j(q')-A_j(q)\|\leq M_j\leq K$ by the mean value theorem for vector-valued functions, where $M_j$ is the uniform bound of $\|(J\phi_j(q))^{-1}J\phi_j(z,w)\|_{op} $ for $D$. Thus, we consider the coordinate chart $(\psi, U)$ such that $v=(0,0)$, where $v=\lim_{j\to\infty}A_j(q)\in\mathbb{CP}^2$. Since we have $\|A_j(z)-A_j(q)\|<K$ for any $z\in D$, $D$ will be mapped uniformly into a bounded neighborhood of $v\in\mathbb{CP}^2$, which finishes the proof by Montel's theorem.
\end{proof}

\begin{lemma}\label{bi}
If $A_j$ defined above is normal and $\Phi$ is the limit of $A_j$, then $\Phi$ is a biholomorphism from $\Omega$ onto $\widehat{\Omega}$.
\end{lemma}
\begin{proof}
By Lemma \ref{gomp} and Proposition \ref{uni}, one immediately observes that $\Phi(\Omega)$ is an open subset in $\mathbb{C}^2$. For simplicity, we denote the open set $\Phi(\Omega)$ by $\Omega'$. Let $\Omega_j=A_j(\Omega)$ and obtain $\widehat{\Omega}$ by the discussion in the beginning of the current section. We are going to prove $\Omega'=\widehat{\Omega}$.

We first prove $\widehat{\Omega}\subset\Omega'$. Otherwise, $\Phi(\Omega)\subsetneq\widehat{\Omega}$, so there exists $p\in\widehat{\Omega}$ and a neighborhood $U\in\widehat{\Omega}$ so that $\Phi(\Omega)\cap U=\varnothing$. Thus, there is a big $N$ such that for any $j>N$, $A_j(\Omega)\cap U=\varnothing$. But this contradicts the definition of $\widehat{\Omega}$.

We now prove $\Omega'\subset\widehat{\Omega}$. If not, there exists a point $p\in\widehat{\Omega}^c$ and a neighborhood $U$ of $p$ in $\widehat{\Omega}$ such that for some $N>0$ and any $j>N$, we have $A_j^{-1}(U)\cap\Omega=\varnothing$. That is, $A_j(\Omega)\cap U=\varnothing$ for any $j>N$. This again contradicts the definition of $\widehat{\Omega}$.

Now we prove injectivity. Suppose we have $\Phi(z')=\Phi(z'')$. Then 
\begin{equation}\label{inj}
\begin{split}
z'-z''&=\Phi_j^{-1}(\Phi_j(z'))-\Phi_j^{-1}(\Phi_j(z''))\\&=\Phi_j^{-1}(\Phi_j(z'))-\Psi(\Phi_j(z'))+\Psi(\Phi_j(z'))-\Psi(\Phi_j(z''))\\&\phantom{{}={}}+\Psi(\Phi_j(z''))-\Phi_j^{-1}(\Phi_j(z'')),
\end{split}
\end{equation}
where $\Psi$ is the limit of $\phi_j^{-1}$. Since $\phi_j^{-1}$ is uniformly bounded, the convergence is not an issue. One can easily see the first and last two terms in the right hand side of Equation (\ref{inj}) vanish as $j$ goes to infinity, while the middle two terms vanish because $\phi(z')=\phi(z'')$. That is, $z'=z''$, which completes the proof.
\end{proof}

\begin{proof}[Proof of Proposition \ref{s1p}]
We have shown if $\|(J\phi_j(q))^{-1}J\phi_j(z,w)\|$ is locally uniformly bounded, then the sequence $A_j$ is normal and thus convergent to a biholomorphism by Lemma \ref{bi}. So it will be enough to show uniform boundness of $\|(J\phi_j(q))^{-1}J\phi_j(z,w)\|$.

This argument is similar to the one Frankel used in \cite{Fr86}. For the sake of completeness, we outline the proof. Let $D\Subset\Omega$ be an arbitrary campact subset. Since $\phi_j$ is normal, there exists a neighborhood $U$ of $p$ in $\mathbb{C}^2$ and $N>0$ such that $\phi_j(D)\Subset U\cap\Omega$ and $U\cap\Omega$ is convex for $j>N$. Thus, it makes sense to define the following map $F_j$ from $D\times D$ to $\Omega$ for $j>N$:
\begin{equation*}
\begin{split}
&(J\phi_j(q))^{-1}\phi_j)^{-1}\circ\bigg(\dfrac{(J\phi_j(q))^{-1}\phi_j(z_1,w_1)+(J\phi_j(q))^{-1}\phi_j(z_2,w_2)}{2}\bigg)\\=&((\phi_j)^{-1}\circ\bigg(\dfrac{\phi_j(z_1,w_1)+\phi_j(z_2,w_2)}{2}\bigg).
\end{split}
\end{equation*}
Let us define $\omega_j=(J\phi_j(q))^{-1}\phi_j$ so that 
\begin{equation}\label{FE}
2\omega_j\circ F_j(x,y)=\omega_j(x)+\omega_j(y).
\end{equation} 
Please note, for simplicity, we denote $(z_1, w_1)$ and $(z_2,w_2)$ by $x$ and $y$, respectively, and $F_j(x,x)=x$. Differentiate both sides of Equation (\ref{FE}) with respect to $x$ and let $y=x$. We obtain that for the $j-th$ term (and we will not write $j$ for simplicity), $2\omega_\alpha(x)F^\alpha_\beta(x,x)=\omega_\beta(x)$, which implies $F^\alpha_\beta=\dfrac{\delta^\alpha_\beta(x,x)}{2}$, where $\delta$ is the Kronecker notation. 

We continue to differentiate both sides of Equation (\ref{FE}) with respect to $x$ and let $y=x$ for the second time. One obtains that $2\omega_{\alpha,\gamma}(x)F^\alpha_{\beta_1}(x,x)F^\gamma_{\beta_2}(x,x)+2\omega_\alpha F^\alpha_{\beta_1,\beta_2}(x,x)=\omega_{\beta_1,\beta_2}(x)$, which gives $\omega_\alpha(x) F^\alpha_{\beta_1,\beta_2}(x,x)=\dfrac{1}{2}\omega_{\beta_1,\beta_2}(x)$. Thus $\|\nabla J\omega_j\|_D<C\|J\omega_j\|_D$, where $\|\cdot\|_D$ denotes the maximum of operator norm over $D$. Since we have $J\omega_j(q)=(J\phi_j(q))^{-1}J\phi_j(q)=\rm Id$, by the comparison theorem of O.D.E, we have $\|J\omega_j(z,w)\|_D<C'$, which completes the proof.
\end{proof}

\begin{remark}\label{nor}
The condition ``locally convex'' guarantees the family of automorphisms is normal and is not removable. Please see the following examples.
\end{remark}

\begin{example}\label{cex1}
$\Omega=\lbrace (z,w)\in\mathbb{C}^2: |z-w^2|^2+|w|^4<1\rbrace$.\\
Let $\displaystyle\phi_j=\Bigg(\frac{(z-w^2)-\alpha_j}{1-\bar{\alpha_j}(z-w^2)}+\sqrt{\frac{1-|\alpha_j|^2}{(1-\bar{\alpha_j}(z-w^2))^2}}, \sqrt[\leftroot{-1}\uproot{12}\scriptstyle 4]{\frac{1-|\alpha_j|^2}{(1-\bar{\alpha_j}(z-w^2)^2)^2}}\Bigg)$ be a family of automorphisms, where $|\alpha_j|\rightarrow 1$ as $j\rightarrow\infty$ (one can show it is not locally convex around the accumulation point $(1,0)$). The Jacobian $J\phi_j$ is $\begin{pmatrix}
c^j_{11}&c^j_{12}\\c^j_{21}&c^j_{22}
\end{pmatrix}$, where 
\begin{equation*}
\begin{split}
&c^j_{11}=\frac{1-|\alpha_j|^2}{(1-\bar{\alpha_j}(z-w^2))^2}-\frac{\bar{\alpha_j}\sqrt{1-|\alpha_j|^2}}{(1-\bar{\alpha_j}(z-w^2))^2}w^2\\&
c^j_{12}=-2\frac{1-|\alpha_j|^2}{(1-\bar{\alpha_j}(z-w^2))^2}w+2\frac{\bar{\alpha_j}\sqrt{1-|\alpha_j|^2}}{(1-\bar{\alpha_j}(z-w^2))^2}w^3+2\frac{\sqrt{1-|\alpha_j|^2}}{1-\bar{\alpha_j}(z-w^2)}w\\&
c^j_{21}=\frac{\bar{\alpha_j}\sqrt[4]{1-|\alpha_k|^2}}{2(1-\bar{\alpha_j}(z-w^2))^\frac{3}{2}}w\\&
c^j_{22}=\frac{2\bar{\alpha_j}\sqrt[4]{1-|\alpha_j|^2}}{(1-\bar{\alpha_j}(z-w^2))^\frac{3}{2}}w^2+\sqrt[\leftroot{-1}\uproot{12}\scriptstyle 4]{\frac{1-|\alpha_j|^2}{(1-\bar{\alpha_j}(z-w^2))^2}}
\end{split}.
\end{equation*}
Also, $(J\phi_j(0,0))^{-1}=\begin{pmatrix}
\dfrac{1}{1-|\alpha_j|^2}&0\\0&\dfrac{1}{\sqrt[4]{1-|\alpha_j|^2}}
\end{pmatrix}$
and one can easily show that $\|(J\phi_j(0,0))^{-1}J\phi_j(z,w)\|$ is not bounded.
\end{example}

With a similar computation as in Example $\ref{cex1}$, the reader can see the same situation in the following domain.

\begin{example}\label{cex2}
$\Omega=\lbrace (z,w)\in\mathbb{C}^2: |z|<1, |w-z^2|<1\rbrace$.
\end{example}

Now we analyze the defining function of $\widehat{\Omega}$. First, we consider the case when both eigenvalues of $J\phi_j$ approach zero. There are two possibilities: $p$ is of finite type or infinite type. 

The following lemma does the job for a finite type point $p$.
\begin{lemma}[the orbit accumulation point case, point is finite type]\label{1}
Let $\Omega$ be a bounded domain in $\mathbb{C}^2$. Assume there is a family of automorphisms $\phi_j=(f_j, g_j)$ and an interior point $q\in\Omega$ such that $\phi_j(q)\rightarrow p$ where $p\in\partial\Omega$. We also assume the normal domain of $\Omega$ at $p$ is locally convex and has a smooth boundary around $p$. If $p$ is a boundary point of finite type, then $\Omega$ is biholomorphic to $v=\rho_k(z,\bar{z})$ where $\rho_k$ is a homogeneous polynomial with degree $k$.
\end{lemma}

\begin{proof}
For the discussion, we refer the reader to \cite{Ki90}.

Since $p$ is finite type, it cannot be in the orbit accumulation variety case. This is because $\partial\Omega$ around $p$ does not contain a non-trivial analytic variety.

Since $p$ is a finite type point, the defining function $v=\rho_k(z,\bar{z},0)+C(z, \bar{z})u+o(u^2)$, where $\lim_{z\to 0}\dfrac{C(z,\bar{z})}{|z|}=0$ (by Taylor theorem and the definition of normal domain) has the homogeneous polynomial $\rho_k(z,\bar{z},0)$ as one component. For simplicity we will write $f(z,\bar{z},u)$ as $f(z,u)$ for any real function $f$. Let 
\begin{equation*}
J\phi_j(q)=\begin{pmatrix} b^j_{11}&b^j_{12}\\b^j_{21}&b^j_{22}\end{pmatrix},
\end{equation*}
and we obtain by computation 
\begin{equation*}
\begin{pmatrix}
b^j_{11}&b^j_{12}\\b^j_{21}&b^j_{22}
\end{pmatrix}
\begin{pmatrix}
z'\\w'
\end{pmatrix}=\begin{pmatrix}
z\\w
\end{pmatrix},
\end{equation*}
where $z'$ and $w'$ are new coordinates after $A_j$.
We assume, without loss of generality, $|b^j_{22}/ b^j_{21}|\geq C>0$. After the $j-th$ step scaling, we obtain the defining function of $\partial\Omega_j$ which becomes $\Im(b^j_{21}z'+b^j_{22}w')=\rho(b^j_{11}z'+b^j_{12}w', \Re(b^j_{21}z'+b^j_{22}w'))$. In $\mathbb{CP}^2$, $\Omega_j$ is bounded by the following boundary points
\begin{equation*}
[b^j_{11}z'+b^j_{12}w', \Re(b^j_{21}z'+b^j_{22}w')+i\rho(b^j_{11}z'+b^j_{12}w', \Re(b^j_{21}z'+b^j_{22}w')), 1]
\end{equation*}  
in homogeneous coordinates, where $b^j_{11}z'+b^j_{12}w'$ is free.

Now, dividing by $b^j_{22}$, we get
\begin{equation*}
\Im(\dfrac{b^j_{21}}{{b^j_{22}}}z'+w')=\dfrac{\rho(b^j_{11}z'+b^j_{12}w', \Re(b^j_{21}z'+b^j_{22}w'))}{{b^j_{22}}}.
\end{equation*}  
Embedding into $\mathbb{C}^3$, one can see $\Omega_j$ approaches the domain bounded by 
\begin{equation}\label{theone}
\lim_{j\to\infty}\Im(\dfrac{b^j_{21}}{{b^j_{22}}}z'+w')=\lim_{j\to\infty}\dfrac{\rho(b^j_{11}z'+b^j_{12}w', \Re(b^j_{21}z'+b^j_{22}w'))}{{b^j_{22}}}
\end{equation} 
in $\mathbb{C}^3$ (also in $\mathbb{C}^2$). Hence $\widehat{\Omega}$ is biholomorphic onto the domain with the boundary defined in Equation (\ref{theone}). 

Thus, it is enough to study the limit of $\partial\Omega_j$. Looking at Equation (\ref{theone}), by the existence of $\partial\widehat{\Omega}$ (otherwise, by \cite{Ki90} $\Omega$ is not bounded), the right hand side must converge to a function, as the left hand side already converges to a function. Otherwise, the limit domain will be $v>\infty$ or $v>0$, neither of which is possible for a bounded domain. Hence, both of $\dfrac{(b^j_{11})^k}{b^j_{22}}$ and $\dfrac{(b^j_{12})^k}{b^j_{22}}$ converge to nonzero constants (see \cite{Ki90}). At the same time, $\rho_k$ converges to a function of $(z,w)$, while the higher terms converge to $0$. Finally we get that $\Omega$ can be defined by $\Im(hz'+w')=\rho_k(cz'+dw',\overline{cz'+dw'})$, where $h=\lim_{j\to\infty}\dfrac{b^j_{21}}{b^j_{22}}$, $c=\lim_{j\to\infty}\dfrac{b^j_{11}}{(b^j_{22})^{1/k}}$ and $d=\lim_{j\to\infty}\dfrac{b^j_{12}}{(b^j_{22})^{1/k}}$. After a change of variable, we finish the proof.
\end{proof}

We will analyze the infinite type $p$ which is not of infinite type II when both of the eigenvalues of $J\phi_j$ go to zero.

The following lemma motivates the key lemma, which makes it possible to analyze the infinite type boundary.

\begin{lemma}
Let $y=f(x)$ be a smooth function of one variable with the graph passing through $(0,0)$, where $x\in(-5,5)$. Let $\lambda_j^i\rightarrow 0$ as $j\rightarrow\infty$ where $\lambda_j^i$ is a nowhere zero sequence and we assume $g(x):=\lim_{j\to\infty}\dfrac{f(\lambda_j^1x)}{\lambda_j^2}$ exists, where $f(x)\in o(x^m)$ for any positive number $m$. Moreover, assume $f(ax)\leq |a|^{m_x}f(x)$ for all $0<a<1$, where $m_x\rightarrow\infty$ as $z\rightarrow 0$. If for a point $x_0\in(-5,5)$ we have $g(x_0)=C>0$, then for any $|x|<|x_0|$, $g(x)=0$. 
\end{lemma}

\begin{proof}
Fix $x,x_0$ so that $|\dfrac{x}{x_0}|<1$. Since $g(x_0)=C>0$, we have, $f(\dfrac{x}{x_0}\cdot \lambda^1_j x_0)\leq |\dfrac{x}{x_0}|^{m_j} f(\lambda^1_jx_0)$, where $m_j\rightarrow\infty$ as $j\rightarrow \infty$. Dividing by $\lambda^2_j$ on both sides, we find that, $g(x)\leq |\dfrac{x}{x_0}|^{m_j} g(x_0)$. Let $j\rightarrow\infty$, we obtain $g(x)=0$ for $|x|<|x_0|$.
\end{proof}

actually, the condition $\lim_{j\to\infty}\dfrac{f(\lambda_j^1x_0)}{\lambda_j^2}=C>0$ for some $x_0$ of the last lemma never happens.

\begin{lemma}\label{dier}
Let $y=f(x)$ be a smooth function passing through $(0,0)$, where $x\in(-5,5)$. Let $\lambda_j^i\rightarrow 0$ as $j\rightarrow\infty$ where $\lambda_j^i$ is nowhere zero sequence. Moreover, assume $f(ax)\leq |a|^{m_x}f(x)$ for all $0<a<1$, where $m_x\rightarrow\infty$ as $z\rightarrow 0$. Then $g(x):=\lim_{j\to\infty}\dfrac{f(\lambda_j^1x)}{\lambda_j^2}=0$ for any $x\in(-5,5)$.
\end{lemma}

\begin{proof}
Suppose there is some $x_0$ such that $\lim_{j\to\infty}\dfrac{f(\lambda_j^1x_0)}{\lambda_j^2}=C>0$. Define $x_j$ such that $|x_j|<|x_0|$ and for arbitrary $j$, we have $|f(\lambda^1_jx_j)-f(\lambda^1_jx_0)|<\dfrac{|\lambda^2_j|}{j}$. This is possible because $\lambda^2_j$ is never zero and $f$ is continuous. But now $\dfrac{|f(\lambda^1_jx_j)-f(\lambda^1_jx_0)|}{|\lambda^2_j|}\rightarrow 0$ as $j\rightarrow\infty$, which contradicts $\lim_{j\to\infty}\dfrac{f(\lambda_j^1x_0)}{\lambda_j^2}=C>0$, by the last lemma.
\end{proof}

Let us return to the multivariable case.

\begin{lemma}\label{disa}
Let $f$ be a real smooth function defined in $\mathbb{C}$ whose graph includes the origin. We assume $b^j_{11}$, $b^j_{12}$, $b^j_{22}$ and $b_j$ approach $0$. If $g(z,w):=\lim_{j\to\infty}\dfrac{f(b^j_{11}z+b^j_{12}w)}{b^j_{22}}$ and $f(ax)\leq |a|^{m_x}f(x)$  for all complex $0<|a|<1$, where $m_x\rightarrow\infty$ as $z\rightarrow 0$, then $\dfrac{f(b^j_{11}z+b^j_{12}w)}{b^j_{22}}\rightarrow 0$ everywhere as $j\rightarrow\infty$.
\end{lemma}

\begin{proof}
We assume the Lemma is incorrect.

One observes 
\begin{equation*}
\frac{f(b^j_{11}z+b^j_{12}w)}{b^j_{22}}=\frac{f((b^j_{11}z_0+b^j_{12}w_0)(\dfrac{b^j_{11}z+b^j_{12}w}{b^j_{11}z_0+b^j_{12}w_0}))}{b^j_{22}}.
\end{equation*} 
Hence, we find $(z,w)$ such that $|b^j_{11}z+b^j_{12}w|<|b^j_{11}z_0+b^j_{12}w_0|$ for large $j$, and we can see $\lim_{j\to\infty}\frac{f(b^j_{11}z+b^j_{12}w)}{b^j_{22}}=0$. We continue to prove that no $(z_0,w_0)$ can satisfy $\lim_{j\to\infty}\frac{f((b^j_{11}z_0+b^j_{12}w_0)(\dfrac{b^j_{11}z+b^j_{12}w}{b^j_{11}z_0+b^j_{12}w_0}))}{b^j_{22}}=C>0$. As in Lemma \ref{dier}, choose a sequence $\lbrace(z_j, w_j)\rbrace$ with limit $(z_0,w_0)$ such that $|b^j_{11}z_j+b^j_{12}w_j|<|b^j_{11}z_0+b^j_{12}w_0|$ and $|f(b^j_{11}z_j+b^j_{12}w_j)-f(b^j_{11}z_0+b^j_{12}w_0)|<\dfrac{|b^j_{22}|}{j}$ by continuity. By a similar argument as before, the assumption that $\dfrac{f(b^j_{11}z_0+b^j_{12}w_0)}{b^j_{22}}$ approaches a nonzero number is impossible. 
\end{proof}

\begin{proposition}[the orbit accumulation point case, point is of infinite type I]\label{2}
Let $\Omega$ be a bounded domain in $\mathbb{C}^2$. Assume there is a family of automorphisms $\phi_j=(f_j, g_j)$ and an interior point $q\in\Omega$ such that $\phi_j(q)\rightarrow p$, where $p\in\partial\Omega$. We also assume the normal domain of $\Omega$ at $p$ is locally convex and has a smooth boundary around $p$. For the orbit accumulation point case, $p$ cannot be a boundary point of infinite type I.
\end{proposition}

\begin{proof}
Let $v>\rho(z,\bar{z},u)$ be the local defining function with $p=(0,0)$. As in the proof of Lemma \ref{1}, we just need to study the limit of 
\begin{equation*}
\frac{1}{b^j_{22}}\rho(b^j_{11}z+b^j_{12}w, \Re(b^j_{21}z+b^j_{22}w)).
\end{equation*}

Since $\Phi$ is biholomorphic from $\Omega$ to $\widehat{\Omega}$, the limit of $\partial\Omega_j$ must also be of infinite type at $(0,0)$. Thus, the limit $\frac{1}{b^j_{22}}\rho(b^j_{11}z+b^j_{12}w, 0)$ is of $o(|(z,w)|^m)$ for any positive integer $m$ (because the higher order terms in Equation (\ref{ust}) below all vanish). For simplicity we also write $z'$ and $w'$ as $z$ and $w$ temporarily. After scaling, again, as in the previous lemma, we obtain 
\begin{equation*}
(\dfrac{b^j_{21}}{b^j_{22}}z+w)=\dfrac{1}{b^j_{22}}\rho(b^j_{11}z+b^j_{12}w, \Re(b^j_{21}z+b^j_{22}w)).
\end{equation*} 
By Taylor's theorem, we have the following, where $C$ is a function satisfying $\lim_{x\to\infty}\dfrac{C(x)}{x}=0$.
\begin{equation}\label{ust}
\begin{split}
&\frac{1}{b^j_{22}}\rho(b^j_{11}z+b^j_{12}w, \Re(b^j_{21}z+b^j_{22}w))\\
=&\frac{1}{b^j_{22}}\rho(b^j_{11}z+b^j_{12}w, 0)+\frac{1}{b^j_{22}}C(b^j_{11}z+b^j_{12}w)\Re(b^j_{21}z+b^j_{22}w)\\&+\frac{1}{b^j_{22}}O(M_x(\Re(b^j_{21}z+b^j_{22}w))^2).
\end{split}
\end{equation}
This implies that the limit of $\dfrac{1}{b^j_{22}}\rho(b^j_{11}z+b^j_{12}w, \Re(b^j_{21}z+b^j_{22}w))$ is $0$ because the first term of the last line of Equation (\ref{ust}) approaches $0$ by Lemma \ref{disa}. Also, the second and third term approach $0$ because of the boundness of $\dfrac{\Re(b^j_{21}z+b^j_{22}w)}{b^j_{22}}$. However, if the limit is $0$, then the domain cannot biholomorphic to a bounded domain (indeed, the domain is biholomorphic to the domain defined by $v=0$).
\end{proof}

If the image of the limit is not a point, i.e. if the orbit accumulate variety case happens, we use the following generalized scaling method.

Fix an interior point $q\in\Omega$ and let $\mathfrak{I}$ be the image of $\phi_0:=\displaystyle\lim_{j\to\infty}\phi_j$. Clearly, $\mathfrak{I}$ has the following property.

\begin{proposition}
$\mathfrak{I}$ contains an (small) analytic disc around any point $p\in\mathfrak{I}$.
\end{proposition}
\begin{proof}
Clearly $\mathfrak{I}$ has complex dimension $1$. Locally around $p$, we can have the constant rank theorem. There exists two local holomorphic coordinate charts $\psi_1$ and $\psi_2$ so that $\psi_1\circ\phi_0\circ\psi_2^{-1}=(z, 0)$, i.e. $\phi_0\circ\psi_2^{-1}=\psi_1^{-1}\circ(z, 0)$, which reveals that the image of $\phi_0$ is an analytic disc locally (since $\psi_1$ and $\psi_2$ are locally defined).
\end{proof}

Now, let us consider the scaling method for the accumulation variety case. We assume $\lambda^1_j(z,w)\rightarrow 0$ but $\lambda^2_j(z,w)\not\rightarrow 0$.

\begin{lemma}
Let $\Omega$ be a bounded domain in $\mathbb{C}^2$. Assume there is a family of automorphisms $\phi_j=(f_j, g_j)$ such that there is an interior point $q\in\Omega$ and a boundary point $p\in\partial\Omega$ with $\phi_j(q)\rightarrow p$. Suppose $\displaystyle\lim_{j\to\infty}J(\phi_j)$ has rank one, i.e., the image of $\lim_{j\to\infty}\phi_j$ contains more than isolated points. Also, let $A_j(z)=B_j^{-1}\phi_j(z,w)$ where $B^{-1}_j=\bigl(\begin{smallmatrix} 1&0\\0&\frac{1}{\lambda^2_j(q)}\end{smallmatrix}\bigr)\bigl(\begin{smallmatrix} u^j_1(q)&u^j_2(q)\\v^j_1(q)&v^j_2(q)\end{smallmatrix}\bigr)$ and $\bigl(\begin{smallmatrix} u^j_1(q)\\u^j_2(q)\end{smallmatrix}\bigr)$, $\bigl(\begin{smallmatrix} v^j_1(q)\\v^j_2(q)\end{smallmatrix}\bigr)$ are the unit eigenvalues of $J\phi_j(q)$. Then $A_j$ converges uniformly on compact subsets to a biholomorphism between $\Omega$ and $\widehat{\Omega}$, where $\widehat{\Omega}$ is the limit of  $A_j(\Omega)$.
\end{lemma}

\begin{remark}
From now on, by $A_j$, we mean the $A_j$ defined in the lemma above.
\end{remark}

\begin{proof}
Since we already showed $J\phi_j(q)\phi_j(z,w)$ is normal, and by computation, we have
\begin{equation*}
\begin{pmatrix} 1&0\\0&\frac{1}{\lambda^2_j(q)}\end{pmatrix}\begin{pmatrix} u^j_1(q)&u^j_2(q)\\v^j_1(q)&v^j_2(q)\end{pmatrix}=\begin{pmatrix} \lambda^1_j(q)&0\\0&1\end{pmatrix}\begin{pmatrix} u^j_1(q)&u^j_2(q)\\v^j_1(q)&v^j_2(q)\end{pmatrix}(J\phi_j(q))^{-1}.
\end{equation*}
To show $A_j$ is normal and converges to an automorphism, it is enough to show
\begin{enumerate}
\item $\|\begin{pmatrix}
\lambda^1_j(q)&0\\0&1\end{pmatrix}\begin{pmatrix} u^j_1(q)&u^j_2(q)\\v^j_1(q)&v^j_2(q)
\end{pmatrix}\|_{op}$ is bounded, and 
\item $\det \begin{pmatrix} u^j_1(q)&u^j_2(q)\\v^j_1(q)&v^j_2(q)\end{pmatrix}$ has a lower nonzero bound.
\end{enumerate}

We can first show that $\|\bigl(\begin{smallmatrix} \lambda^1_j(q)&0\\0&1\end{smallmatrix}\bigr)\bigl(\begin{smallmatrix} u^j_1(q)&u^j_2(q)\\v^j_1(q)&v^j_2(q)\end{smallmatrix}\bigr)\|_{op}$ is bounded. This is true because the norm of the matrix, which is equivalent to the operator norm, equals 
\begin{equation*}
\sqrt{u^j_1(q)^2+u^j_2(q)^2+v^j_1(q)^2+v^j_2(q)^2}=\sqrt{2},
\end{equation*}
and $\|\begin{pmatrix}
\lambda^1_j(q)&0\\0&1
\end{pmatrix}\|$ is also bounded.

Next, we show $|\det \bigl(\begin{smallmatrix} u^j_1(q)&u^j_2(q)\\v^j_1(q)&v^j_2(q)\end{smallmatrix}\bigr)|$ has a lower nonzero bound. For this aim, it is enough to show that the limit $\begin{pmatrix} u^0_1(q)\\u^0_2(q)\end{pmatrix}$ and $\begin{pmatrix}v^0_1(q)\\v^0_2(q)\end{pmatrix}$ are linearly independent. Assume they are not, so without loss of generality, $\begin{pmatrix}
u^j_1(q)\\u^j_2(q)
\end{pmatrix}-\begin{pmatrix}
v^j_1(q)\\v^j_2(q)
\end{pmatrix}\rightarrow 0$ as $j\rightarrow\infty$. We also have $J\phi_j(q)\begin{pmatrix} u^j_1(q)\\u^j_2(q)\end{pmatrix}=\lambda^j_1\begin{pmatrix} u^j_1(q)\\u^1_j(q)\end{pmatrix}$ and $J\phi_j(q)\begin{pmatrix} v^j_1(q)\\v^j_2(q)\end{pmatrix}=\lambda^2_j\begin{pmatrix} v^j_1(q)\\v^j_2(q)\end{pmatrix}$, and then $J\phi_j(q)\begin{pmatrix} u^j_1(q)-v^j_1(q)\\u^j_2(q)-v^j_2(q)\end{pmatrix}=\lambda^1_j\begin{pmatrix} u^j_1(q)\\u^j_2(q)\end{pmatrix}-\lambda^2_j\begin{pmatrix} v^j_1(q)\\v^j_2(q)\end{pmatrix}$. On the left hand side, since $J\phi_j(q)$ is a bounded operator, $\lambda^1_j\begin{pmatrix} u^j_1(q)\\u^j_2(q)\end{pmatrix}-\lambda^2_j\begin{pmatrix} v^j_1(q)\\v^j_2(q)\end{pmatrix}\rightarrow 0$ which is impossible (because $\lambda^2_j\rightarrow 0$ while $\lambda^1_j\rightarrow C\in\mathbb{C}$ where $C$ is a nonzero number).
\end{proof}

\begin{lemma}[the orbit accumulation variety case]\label{3}
Let $\Omega$ be a bounded domain in $\mathbb{C}^2$. Assume there is a family of automorphisms $\phi_j=(f_j, g_j)$ and an interior point $q\in\Omega$ such that $\phi_j(q)\rightarrow p$ where $p\in\partial\Omega$. We also assume the normal domain of $\Omega$ at $p$ is locally convex and has a smooth boundary around $p$. If the case is of the orbit accumulation variety case, then $\Omega$ is biholomorphic to $\mathbb{D}\rtimes_\theta\mathbb{H}^+$.
\end{lemma}

\begin{proof}
Locate the domain such that $\phi_j(q)\rightarrow 0$ as $j\rightarrow\infty$ such that $\mathfrak{I}$ has tangent complex line $(z,0)$ at $(0,0)$ where $z\in\mathbb{C}$.
As before, let us observe the limit of the family of biholomorphisms
\begin{equation}\label{ho}
\begin{pmatrix}
1 & 0\\0& \frac{1}{\lambda^2_j(q)}
\end{pmatrix}
\begin{pmatrix}
u^j_1(q) &u^j_2(q)\\ v^j_1(q)& v^j_2(q)
\end{pmatrix}
\begin{pmatrix}
f_j(z,w)\\g_j(z,w)
\end{pmatrix}.
\end{equation}
We will also denote $\lambda^2_j(q)$ by $\lambda_j$ in this proof for simplicity.
Since $\begin{pmatrix}
u^j_1(q) \\ u^j_2(q)\end{pmatrix}$ and $\begin{pmatrix}
v^j_1(q) \\ v^j_2(q)
\end{pmatrix}$ are unit vectors, they have limits $\begin{pmatrix}
u^0_1(q) \\ u^0_2(q)\end{pmatrix}$ and $\begin{pmatrix}
v^0_1(q) \\ v^0_2(q)
\end{pmatrix}$. Thus, one observes that $v^0_1(q)z_0+v^0_2(q)w_0=0$ where $(z_0, w_0)\in\mathfrak{I}$ by letting $j\rightarrow\infty$ in Equation (\ref{ho}).

Hence, without loss of generality, we can assume $v^0_2(q)\neq 0$ (because otherwise $v^0_2(q)=v^0_1(q)=0$, since $z_0\not\equiv 0$ by the assumption.), and we apply the holomorphic transform $(z,w)\mapsto (z, w-\dfrac{v^0_1z}{v^0_2})$. As a result, we can always assume $\mathfrak{I}$ is contained in $\mathbb{C}\times\lbrace0\rbrace$. However, our domain is already a normal domain around $(0,0)$. After observation of the complex tangent plane at $(0,0)$, we find that $|v^0_2(q)|=1$ and $v^0_1(q)=0$.

Consider $(\phi^0)^{-1}(z_0,0)$, where $(z_0,0)\in\mathfrak{I}$. One can observe that after applying $A_j$ (we denote $(z',w')$ the new coordinates after changing variables $A_j$), we have $z'=u^j_1(q)f_j(z,w)+u^j_2(q)g_j(z,w)\rightarrow u^0_1(q)z_0$, where $(z_0, 0)\in\mathfrak{I}$. Note $u^0_1(q)\neq 0$, because otherwise $\begin{pmatrix}
u^0_1(q)&u^0_2(q)\\v^0_1(q)&v^0_2(q)
\end{pmatrix}$ has rank at most $1$, given $v^0_1(q)=0$.

Now, let us analyze the boundary after scaling as before. We assume 
\begin{equation*}
\begin{pmatrix}
b^j_{11}&b^j_{12}\\b^j_{21}&b^j_{22}
\end{pmatrix}^{-1}=\begin{pmatrix}
u^j_1(q)&u^j_2(q)\\v^j_1(q)&v^j_2(q)
\end{pmatrix}.
\end{equation*}
One can observe that $b^j_{21}\rightarrow 0$ because of $v^0_1(q)=0$.
We have the following transform
\begin{equation*}
\begin{split}
f_j(z,w)=b^j_{11}z'+b^j_{12}\lambda_jw',\\
g_j(z,w)=b^j_{21}z'+b^j_{22}\lambda_jw'
\end{split}.
\end{equation*}
Without loss of generality, let $0=\alpha+k_1\Im\xi+k_2\Re\xi+o(|\xi|)$ be the defining function of the projection of $(\phi^0)^{-1}(z_0,0)$ to the $w$-plane. By our assumption, $g_j(z,w)\rightarrow 0$ and then $\alpha=0$. The resulting defining function will be a straight line passing through the limit of $-\dfrac{b^j_{21}}{b^j_{22}\lambda_j}z'$ in the $w$-plane which is also holomorphic in $z'$, where $z'=u^0_1(q)z_0$, $(z_0,0)\in\mathfrak{I}$. Note here, $|\dfrac{b^j_{21}}{b^j_{22}\lambda_j}|$ is bounded otherwise, $-\dfrac{b^j_{21}}{b^j_{22}\lambda_j}z'$ are infinity for those  $z'\neq0$ and this is impossible because it will not form a boundary of a domain.

So, $\Omega$ is biholomorphic to $\mathfrak{I}'\rtimes_\theta\mathbb{H}^+$, where $\mathfrak{I}'=\lbrace (z', w'): z'=u^0_1(q)z_0, w'=-\lim_{j\to\infty}\dfrac{b^j_{21}}{b^j_{22}\lambda_j}z'\rbrace$.

Thus $\mathfrak{I}'$ is a Riemann surface which is biholomorphic to the unit disc in $\mathbb{C}^2$, because $\mathfrak{I}'$ is biholomorphic to $\mathfrak{I}$ which is bounded in $\mathbb{C}$ (thanks to the uniformization theorem). After a biholomorphic mapping $(h, \rm Id)$, where $h$ is the biholomorphic mapping which maps the Riemann surface onto unit disc $\mathbb{D}$ by the uniformization theorem and $\rm Id$ is the identity map, we obtain that the original domain is biholomorphic to $\mathbb{D}\rtimes_\theta\mathbb{H}^+$.
\end{proof}

\begin{proof}[Proof of Theorem \ref{main}]
By Proposition \ref{s1p}, we obtain a domain $\widehat{\Omega}$ by analyzing its defining function in both the orbit accumulation point case and the orbit accumulation variety case. The proof can be completed by Lemma \ref{1}, Proposition \ref{2} and Lemma \ref{3}.
\end{proof}

\begin{corollary}\label{fs}
Let $\Omega$ be a bounded domain in $\mathbb{C}^2$. Assume there is a family of automorphisms $\phi_j=(f_j, g_j)$ and an interior point $q\in\Omega$ such that $\phi_j(q)\rightarrow p$, where $p\in\partial\Omega$. We also assume the normal domain of $\Omega$ at $p$ is locally convex and has a smooth boundary around $p$. Then $\rm Aut(\Omega)$ contains $\mathbb{R}$.
\end{corollary}

\begin{remark}
Corollary \ref{fs} partially confirms Question 3.17 of \cite{FS01} in the case of bounded ``locally convex'' domain in $\mathbb{C}^2$.
\end{remark}

\begin{proof}
For the orbit accumulation point case, the translation is defined as $L_t:(z,w)\mapsto(z, w+t)$, where $t\in\mathbb{R}$. Because of the scaling method of Frankel and Kim, we can remove the term $O((\Re w)^2)+O(z\Re w)$ from $\Im w=\rho(z,\bar{z},0)+O((\Re w)^2)+O(z\Re w)$ no matter what the type of $p$ is. Moreover, $L_t$ is an automorphism for all domains of the form $\Im w=\rho(z, \bar{z})$. For the orbit accumulation variety case, let $L_t(z,w)=(z, \dfrac{w}{e^t})$. Then $L_t\in\rm Aut(\widehat{\Omega})$ and $L_t$ is isomorphic to $\mathbb{R}$.
\end{proof}

\section{A remark about Greene-Krantz conjecture}\label{sec3}
In this section we build the connection between our results and the following conjecture.
\begin{conjecture*}[Greene-Krantz]
Let $\Omega$ be a bounded domain with a smooth boundary in $\mathbb{C}^n$. Assume there is a family of automorphisms $\phi_j=(f_j, g_j)$ and an interior point $q\in\Omega$ such that $\phi_j(q)\rightarrow p$, where $p\in\partial\Omega$. Then $p$ is of finite type.
\end{conjecture*}

Although $\mathbb{D}\rtimes_\theta\mathbb{H}^+$ is very possible not to biholomorphic to any bounded domain with a smooth boundary, we can not prove it here now. One possible method is to generalize the theorem in \cite{CV13} from product domains to ``$\rtimes_\theta$'' domains.

In this section, we show $\mathbb{D}\times\mathbb{H}^+$ (the bidisc) is not biholomorphic with any bounded domain with a smooth boundary in $\mathbb{C}^2$.

\begin{theorem}
$\mathbb{D}\times\mathbb{H}^+$ is not biholomorphic with any bounded doman with a smooth boundary in $\mathbb{C}^2$.
\end{theorem}
\begin{proof}
$\mathbb{D}\times\mathbb{H}^+$ is biholomorphic to $\mathbb{D}\times\mathbb{D}$ by a Cayley transform on the second variable. We observe $\mathbb{D}$ is pseudoconvex and satisfies condition $R$. Suppose there is a biholomorphism $f$ which maps the bidisc $\mathbb{D}\times\mathbb{D}$ onto a bounded domain $\Omega$ with a (globally) smooth boundary. Then $f$ extends smoothly up to boundary of the bidisc and $\Omega$ by Theorem 1.1 in \cite{CV13}. However, the bidisc does not have a smooth boundary, which contradicts the extension of $f$.
\end{proof}

\begin{corollary}\label{GK}
Let $\Omega$ be a bounded domain of smooth boundary in $\mathbb{C}^2$. Assume there is a family of automorphisms $\phi_j=(f_j, g_j)$ and an interior point $q\in\Omega$ such that $\phi_j(q)\rightarrow p$, where $p\in\partial\Omega$ is not of infinite type II. We also assume the normal domain of $\Omega$ at $p$ is locally convex around $p$. Then either $p$ is finite type or $\Omega$ is biholomorphic to $\mathbb{D}\rtimes_\theta\mathbb{H}^+$, where $\theta\not\equiv 0$. 
\end{corollary}

\begin{remark}
In this note, every result is under the assumption ``the boundary is locally bounded around accumulation points''. It will be very progressive if one can remove this condition. Also, whether $\mathbb{D}\rtimes_\theta\mathbb{H}^+$ biholomorphic to a bounded domain with a (globally) smooth boundary is an interesting question. The existence of a domain with a boundary of infinite type II should also be studied. Once one overcomes all of obstacles, it will lead him/her to the complete proof of the Greene-Krantz conjecture in $\mathbb{C}^2$.
\end{remark}

\bigskip
\bigskip
\noindent {\bf Acknowledgments}. I thank my advisor Prof. Steven Krantz a lot for always being patient about answering my every (even stupid) question. He also carefully read this note and made a lot of extremely useful comments. I also thank Prof. Guido Weiss who was interested in my work and has always been helpful to me. Specifically, he was very patient to listen to the presentation about the note.  Moreover, I have profited very much from the discussion with Prof. N. Mohan Kumar and a special thanks goes to him. Last, but not least, I appreciate Prof. Pascal Thomas for discussing the Kobayashi metric with me during the 39th Spring Lecture in Arkansas.

\bibliographystyle{plainnat}
\bibliography{mybibliography}
\end{document}